\documentclass{article}
\usepackage{amsmath}
\usepackage{amsfonts}
\usepackage{amsthm}
\usepackage{amssymb}
\usepackage{color}

\newtheorem{theorem}{Theorem}%[section]
\newtheorem{lemma}[theorem]{Lemma}

\newtheorem{observation}[theorem]{Observation}

\theoremstyle{definition}
\newtheorem{definition}[theorem]{Definition}
\theoremstyle{remark}
\newtheorem{remark}[theorem]{Remark}

\def\f2{\mathbb{F}_2}

\def\lip{\hskip0.02cm{\rm Lip}\hskip0.01cm}

\newcommand{\diam}{{\rm diam}\hskip0.03cm}

\begin{document}
\title{\LARGE{Low-distortion embeddings of graphs with large girth}}

\author{Mikhail~I.~Ostrovskii\footnote{Participant, NSF supported Workshop in Analysis and Probability, Texas A \&\ M
University}}

\date{\today}
\maketitle

\noindent{\bf Abstract.} The main purpose of the paper is to
construct a sequence of graphs of constant degree with
indefinitely growing girths admitting embeddings into $\ell_1$
with uniformly bounded distortions. This result solves the problem
posed by N.~Linial, A.~Magen, and A.~Naor (2002).
\medskip


\begin{large}

\begin{definition}\label{D:lip} Let $C<\infty$. A map $f:
(X,d_X)\to (Y,d_Y)$ between two metric spaces is called
$C$-\emph{Lipschitz}\index{$C$-Lipschitz map} if \[\forall u,v\in
X\quad d_Y(f(u),f(v))\le Cd_X(u,v).\] A map $f$ is called
\emph{Lipschitz}\index{Lipschitz map} if it is $C$-Lipschitz for
some $C<\infty$. For a Lipschitz map $f$ we define its
\emph{Lipschitz constant}\index{Lipschitz constant} by
\[\lip f:=\sup_{d_X(u,v)\ne 0}\frac{d_Y(f(u),
f(v))}{d_X(u,v)}.\]

A map $f:X\to Y$ is called a \emph{$C$-bilipschitz
embedding}\index{$C$-bilipschitz embedding} if there exists $r>0$
such that
\begin{equation}\label{E:MapDist}\forall u,v\in X\quad rd_X(u,v)\le
d_Y(f(u),f(v))\le rCd_X(u,v).\end{equation} A \emph{bilipschitz
embedding}\index{bilipschitz embedding} is an embedding which is
$C$-bilipschitz for some $C<\infty$. The smallest constant $C$ for
which there exist $r>0$ such that \eqref{E:MapDist} is satisfied
is called the \emph{distortion}\index{distortion} of $f$. (It is
easy to see that such smallest constant exists.)
\medskip

The infimum of distortions of all embeddings of a finite metric
space $X$ into the Banach space $\ell_1$ is called the {\it
$\ell_1$ distortion} of $X$ and is denoted $c_1(X)$.
\end{definition}

The $\ell_1$ distortion of finite metric spaces plays an important
role in the theory of approximation algorithms, see \cite{Lin02},
\cite{LLR95}, \cite{Mat02}, \cite{Mat05}, and \cite{Nao10}.
\medskip

Our main purpose is to solve the following problem suggested in
\cite[p.~393]{LMN02} and repeated in \cite[Open Problem 7]{Lin02}
and  \cite[Problem 2.3]{Mat10}: Does there exist a sequence of
$k$-regular graphs, $k\ge 3$, with indefinitely growing girths and
uniformly bounded $\ell_1$ distortions? (All graphs mentioned in
this paper are endowed with their shortest path distance.) We are
going to show that such sequences exist.\medskip

The construction of this paper is inspired by the paper
\cite{AGS11+}. Recall that the {\it girth} $g(G)$ of a graph $G$
is the length of a shortest cycle in $G$. We start with a sequence
of $k$-regular graphs $\{G_n\}$ with indefinitely increasing
girths $g(G_n)$, such that
\begin{equation}\label{E:GirthMax}g(G_n)\ge c\,\diam(G_n)\end{equation} for some absolute constant
$c$. Existence of such sequences of graphs is known for long time,
see \cite[Chapter III, \S 1]{Bol78}. In the 1980s the constants
involved in the construction were significantly improved, see
\cite{Mar82}, \cite{Imr84}, and \cite{LPS88}. For each graph $G$
in the sequence $\{G_n\}$ we consider its lift $\widetilde G$ in
the sense of the papers \cite{AL06} and \cite{DL06}. (We would
like to warn the reader that somewhat different terminology (graph
covers, voltage graphs) is used in other publications on the
topic, such as \cite{AGS11+}, \cite{GT77}, and \cite{GT87}.) The
particular version of the lift which we use is the same as the
lift used in \cite{AGS11+}, but it is applied to a different
sequence of graphs. Also we need somewhat stronger estimates than
those which were sufficient for \cite{AGS11+}. Another difference
of our presentation from the presentation in \cite{AGS11+} is that
we try to keep the presentation as elementary as possible, without
assuming any topological and group-theoretical background of the
reader. We use only some basic notions of graph theory and the
definition of the space $\ell_1$. We hope that our
graph-theoretical terminology is standard, readers can find all
unexplained terminology in \cite{BM08}.

\begin{definition}\label{L:lift} Let $L$ be a finite set.
A {\it lift} $\widetilde G$ of a graph $G=(V(G), E(G))$ is a graph
with vertex set $V(\widetilde G)=V(G)\times L$. The edge set
$\widetilde G$ is the union of perfect matchings corresponding to
edges of $E(G)$. The matching corresponding to an edge $uv$
matches $\{u\}\times L$ with $\{v\}\times L$.
\end{definition}

Definition \ref{L:lift} immediately implies that there are
well-defined projections $E(\widetilde G)\to E(G)$ and
$V(\widetilde G)\to V(G)$: edges of the matching corresponding to
$uv$ are projected onto $uv$ and vertices of $\{u\}\times L$ are
projected onto $u$. We denote both of the projections by $\pi$. It
is clear from the definition that the degrees of all vertices in
$\widetilde G$ whose projection in $G$ is $u$ are the same as the
degree of $u$. In particular, any lift of a $k$-regular graph is
$k$-regular.\medskip

\begin{remark}\label{R:WalkLifting} It is easy to see that for each walk $\{e_i\}_{i=1}^n$ in $G$
and each vertex $\widetilde u\in V(\widetilde G)$ of the form
$\widetilde u=(u,\ell)$ with $\ell\in L$ and $u$ being the initial
vertex of the walk $\{e_i\}_{i=1}^n$; there is a uniquely
determined {\it lifted walk} $\{\widetilde e_i\}_{i=1}^n$ in
$\widetilde G$ for which $\pi(\widetilde e_i)=e_i$ and $\widetilde
u$ is the initial vertex.
\end{remark}

\begin{remark}\label{R:backtrack}
It is clear that if a walk in $G$ has an edge $e$ which is
backtracked (that is, the walk contains two consecutive edges
$e$), then the corresponding edge in the lifted walk is also
backtracked.
\end{remark}

Remark \ref{R:backtrack} implies that the projection of a cycle in
$\widetilde G$ to $G$ cannot be such that its edges induce in $G$
a subgraph having vertices of degree $1$. In particular, the graph
induced by edges of the projection of a cycle in $\widetilde G$
contains cycles in $G$. This immediately implies $g(\widetilde
G)\ge g(G)$.
\medskip

We apply the lift construction to the graphs $\{G_n\}$ mentioned
above. The fact that we get $k$-regular graphs with indefinitely
increasing girths follows immediately from the observations which
we just made. It remains to specify lifts for which there are
suitable estimates for $\ell_1$ distortions of the obtained
graphs. The bounds for the distortions which we get are in terms
of the constant $c$ in \eqref{E:GirthMax}.
\medskip

For each $G\in\{G_n\}_{n=1}^\infty$ we do the following. We pick a
spanning tree $T$ in $G$ and let $S$ be the set of edges of $G$
which are not in $T$. We let $L$ to be the set $\{0,1\}^S$, so
each element of $L$ can be regarded as a $\{0,1\}$-valued function
on $S$. For each $uv\in E(G)$ we need to specify a perfect
matching of $u\times L$ and $v\times L$. To specify the perfect
matching it suffices, for each edge in $E(G)$, to pick a bijection
of the set $L$. We do this in the following way:

\begin{itemize}

\item If $e\in E(T)$ (that is, if edge is in the spanning tree
which we selected), then the corresponding bijection is the
identical mapping on $L$.

\item If $e\in S$, then the bijection maps each function $f$ on
$S$ to the function $h$, which has the same values as $f$
everywhere except the edge $e$, and on the edge $e$ its value is
the other one (recall that we consider $\{0,1\}$-valued
functions).

\end{itemize}

We denote the graphs obtained from $\{G_n\}$ using such lifts by
$\{\widetilde G_n\}$. The following theorem is the main result of
this paper.

\begin{theorem}\label{T:main} $c_1(\widetilde G_n)=O(1)$.
\end{theorem}

The main steps in our proof are presented as lemmas.

\begin{lemma}\label{L:cuts} For each edge $e\in E(G)$ the set of all edges $\widetilde e\in
E(\widetilde G)$ for which $\pi(\widetilde e)=e$ forms an edge cut
in $\widetilde G$.\end{lemma}

\begin{proof} The statement is simple for $e\in S$. In this case
it is quite easy to describe the sets separated by the cut: they
are the sets $V(G)\times A_{e,0}$ and $V(G)\times A_{e,1}$ where
$A_{e,0}$ and $A_{e,1}$ are the sets of functions in $\{0,1\}^S$
whose values on $e$ are equal to $0$ and $1$,
respectively.\medskip

As for edges corresponding to the tree we have the following: an
edge $e\in E(T)$ splits the tree $T$ into two components, call
them $A$ and $B$. The edges of $S$ are either within one of the
components $A$ and $B$, or between them. We consider two sets of
vertices:
\medskip

The set $P_1$ consisting of pairs $(v,f)$ where either $v\in A$
and the sum of values of $f$ corresponding to edges in $S$ passing
from $A$ to $B$ (recall that $e\notin S$) is even, or $v\in B$ and
the sum of values of $f$ corresponding to edges in $S$ passing
from $A$ to $B$ is odd.
\medskip

The set $P_2$ consisting of pairs $(v,f)$ where either $v\in A$
and the sum of values of $f$ corresponding to edges in $S$ passing
from $A$ to $B$ is odd, or $v\in B$ and the sum of values of $f$
corresponding to edges in $S$ passing from $A$ to $B$ is
even.\medskip

It is easy to check that edges connecting $P_1$ and $P_2$ in
$\widetilde G$ are those and only those edges whose projection to
$E(G)$ is $e$.
\end{proof}

\begin{remark} In \cite[Lemma 3.3]{AGS11+} it was observed that the more difficult part in the proof of  Lemma
\ref{L:cuts} follows from the easier part and a certain
universality result on graph lifts.
\end{remark}

Now we are ready to introduce an embedding $F$ of $V(\widetilde
G)$ into $\mathbb{R}^{E(G)}$ (the set of real-valued functions on
$E(G)$), which has the desired property: the distortion of $F$ is
bounded above by a universal constant if we endow
$\mathbb{R}^{E(G)}$ with its $\ell_1$-norm.\medskip

For each edge cut $R(e)$ defined by the set of edges $\widetilde
e$ in $\widetilde G$ satisfying $\pi(\widetilde e)=e$, we call one
of the sides of the cut $R(e)$ the {\it $0$-side}, and the other
side the {\it $1$-side} and introduce a $\{0,1\}$-valued function
$h_e$ on $V(\widetilde G)$ given by $h_e(x)=0$ if $x$ is in the
$0$-side of $R(e)$, and by $h_e(x)=1$ if $x$ is in the $1$-side of
$R(e)$.
\medskip

If we endow $\mathbb{R}^{E(G)}$ with its $\ell_1$-norm the
Lipschitz constant of this embedding is $1$. In fact, the cuts
$R(e)$ are disjoint and each edge of $\widetilde G$ is in exactly
one of the cuts. Therefore $||F(x)-F(y)||_1=1$ if $x$ and $y$ are
adjacent vertices of $\widetilde G$.\medskip

To estimate the Lipschitz constant of $F^{-1}$ we consider $x,y\in
V(\widetilde G)$ and observe the following:

\begin{observation}
If $P$ is a path between $x$ and $y$ in $\widetilde G$, then
$d_{\widetilde G}(x,y)\le \hbox{\rm length}(P)$ and
$||F(x)-F(y)||_1$ is the number of edges in the walk $\pi(P)$
which are repeated in the walk an odd number of times.
\end{observation}

This observation shows that we can prove the statement about the
distortion if, for each pair $x,y$ of vertices in $\widetilde G$
we can show that a shortest $xy$-path $P$ in $\widetilde G$ is
such that sufficiently many of the edges in the walk $\pi(P)$ are
repeated only one time and other edges are repeated at most two
times.\medskip

Remark \ref{R:backtrack} (on backtracking) implies that a vertex
in the subgraph of $G$ induced by edges of $\pi(P)$ can have
degree $1$ if and only if it is the projection of either the
beginning or the end of the path $P$.\medskip

\begin{lemma}\label{L:repetitions} Let $x,y\in V(\widetilde G)$ and let $P$ be a shortest $xy$-path.
Let $I(P)$ be the subgraph of $G$ induced by edges of $\pi(P)$.
Only cut edges of $I(P)$ can be repeated in the walk $\pi(P)$. Cut
edges of $I(P)$ cannot be repeated in the walk $\pi(P)$ more than
twice.
\end{lemma}

\begin{proof} It is convenient to consider a non-simple graph
$N(P)$ having $I(P)$ as its underlying simple graph and having as
many parallel edges for each edge of $I(P)$, as many times the
edge is repeated in $\pi(P)$. It is easy to see that $N(P)$ has an
Euler trail starting at $\pi(x)$ and ending at $\pi(y)$ (we just
follow the projection $\pi(P)$, using different parallel edges
instead of repeating edges of the underlying graph).
\medskip

To prove the first statement of the lemma, we assume the contrary,
that is, there is an edge $e$ in $I(P)$ which is not a cut edge,
but is repeated in $\pi(P)$. This assumption implies that if we
delete from $N(P)$ two edges parallel to $e$, the result, which we
denote $N'$, will still be a connected graph.
\medskip

By the well-known characterization of graphs having Euler trails,
degrees of all vertices of $N(P)$, except possibly $\pi(x)$ and
$\pi(y)$, are even. It is clear that this condition still holds
for the degrees of $N'$. Using the well-known characterization of
graphs having Euler trails again, we get that the remaining graph
contains an Euler trail which starts at $\pi(x)$ and ends at
$\pi(y)$.
\medskip

We claim that the lift of this trail, if we start the lift at $x$,
will end at $y$, thus giving a shorter $xy$-path and leading to a
contradiction.
\medskip

To prove the claim, we observe that our construction of the lift
of $G$ and our definition of a lifted walk (see Remark
\ref{R:WalkLifting}) are such that the change in the
$L$-coordinate in each step (when we walk along the lifted walk)
is made only in one value of the corresponding $\{0,1\}$-valued
function on $S$, the choice of this coordinate depends only on the
$\pi$-projection of the edge which we are passing, and not on the
direction in which we pass it, or on the $L$-coordinate of the
vertex we are at (this is a very important property of the graph
lift which we consider). Also, we need an obvious observation that
if we change some value of a $\{0,1\}$-valued function twice, it
returns to its original value. Hence the total change in the
$L$-coordinate as we walk along the lift of the Euler trail of
$N'$ is the same as for the original Euler trail in $N(P)$
(formally speaking, we need to replace Euler trails in $N'$ and
$N(P)$ by the corresponding walks in the underlying simple graph
$I(P)$). Hence we end up at $y$.
\medskip

We can get a contradiction in the same way if we assume that some
of the edges in $\pi(P)$ are repeated more than twice. Proving the
first statement we used the assumption that $e$ is not a cut edge
only once: when we claimed that removing two copies of $e$ from
$N(P)$ we get a connected graph. For the second statement we use
the following trivial observation instead: if there is a triple of
parallel edges in $N(P)$, deletion of two of them does not
disconnect the graph.
\end{proof}

Lemma \ref{L:repetitions} shows that to complete the proof of the
theorem it remains to show that the number of cut edges in $I(P)$
(where $P$ is a shortest $xy$-path) which are repeated twice in
$\pi(P)$ cannot be much larger than the number of the remaining
edges. To show this we consider two types of subgraphs in  $I(P)$:

\begin{itemize}

\item[{\bf (i)}] Maximal $2$-edge-connected subgraphs. Let $C$ be
the number of such subgraphs.

\item[{\bf (ii)}] Maximal subgraphs satisfying the conditions: (1)
They are paths; (2) All internal vertices of these paths (if any)
have degree $2$ in $I(P)$; (3) All edges of these paths are cut
edges of $I(P)$. Let $N$ be a number of such subgraphs.

\end{itemize}

\begin{observation}\label{O:card(i)} The number of edges in each subgraph of the type {\bf (i)}
is at least $g(G)$. This statement is easy to see, because each
such subgraph contains a cycle in $G$.
\end{observation}

\begin{lemma}\label{L:quantity} $N\le 2C+1$.
\end{lemma}

\begin{proof} We contract each maximal $2$-edge-connected subgraph to a
vertex, and denote the obtained graph by $D_1$. It is clear that
$D_1$ is a tree, and each subgraph of type {\bf (ii)} is mapped
into $D_1$ isomorphically, and has properties (A) It is a path;
(B) All internal vertices of this path (if any) have degree $2$ in
$D_1$; (C) All edges of these paths are cut edges of $D_1$ (these
properties are analogues of (1), (2), and (3) described in {\bf
(ii)} for $D_1$). On the other hand, maximality can be lost. The
maximality is lost in the cases where there is a maximal
$2$-edge-connected subgraphs of $I(P)$ incident with exactly $2$
cut edges of $I(P)$. Denote the number of such maximal
$2$-edge-connected subgraphs by $H$.
\medskip

Denote by $N_1$ the number of maximal subgraphs in $D_1$ having
properties (A), (B), and (C), the previous paragraph implies that
we have $N-N_1=H$.
\medskip

Observe that, by Remark \ref{R:backtrack} (on backtracking), all,
except possibly two, of leaves of $D_1$ correspond to maximal
$2$-edge-connected subgraphs, so we need to estimate the number of
these leaves, let us denote it by $J$.\medskip

One of the ways to do this is to replace all paths with internal
vertices of degree $2$ in $D_1$ by edges, and denote the obtained
tree by $D_2$. The number of edges in $D_2$ is equal to $N_1$, and
all vertices in $D_2$ which are not leaves have degrees at least
$3$. Also $D_1$ and $D_2$ have the same number of leaves.\medskip

So we need to estimate from below the number $J$ of leaves in a
graph with $N_1$ edges and all vertices which are not leaves
having degrees at least $3$. Counting the sum of all degrees of
$D_2$ in two ways we get
\[2N_1\ge 3(N_1-J+1)+J,
\]
or $2J\ge N_1+3$. We have $2C\ge 2(J-2)+2H\ge 2J+H-4\ge
N_1+H-1=N-1$, which is the desired inequality.
\end{proof}

\begin{lemma}\label{L:length} Let $K$ be a path in $I(P)$ satisfying the conditions: {\rm (a)} All of its internal
vertices have degree $2$ in $I(P)$; {\bf (b)} Each of its edges is
repeated twice in the walk $\pi(P)$, where $P$ is a shortest
$xy$-path in $\widetilde G$ $(x,y\in V(\widetilde G))$. Then the
length of $K$ is $\le\diam G$.
\end{lemma}

\begin{proof} Let $u,v$ be the ends of $K$. If the length of  $K$ is more than
$\diam G$, then there is a strictly shorter $uv$-path $K'$ in $G$,
we are going to use this path to construct a shorter than $P$ path
in $\widetilde G$ joining $x$ and $y$. We do this in the most
straightforward way: first we modify the walk $\pi(P)$ in the
following way: each time when we walk through $K$, we walk through
$K'$ instead. It remains to show that if we lift this walk to
$\widetilde G$, we get another $xy$-walk.\medskip

In fact, this new walk clearly starts at $\pi(x)$ and ends at
$\pi(y)$. We need only to check that the $L$-coordinate at the end
of the walk will be the same as for the original walk. This
follows immediately from the observation that we made earlier: if
we walk through two edges with the same $\pi$-projection twice,
the corresponding changes in the $L$-coordinate cancel each other.
Since this happens for each edge of both $K$ and $K'$, the
$L$-coordinates corresponding to $\pi(y)$ at the end of the walks
will be the same for lifts of both walks. This proves the lemma.
\end{proof}

\begin{proof}[Proof of Theorem \ref{T:main}] We consider two cases
separately:
\medskip

Case 1. $C=0$. By Lemma \ref{L:quantity}, $N=1$ in this case. We
get, by Remark \ref{R:backtrack} (on backtracking), that each edge
of $I(P)$ is used in the walk $\pi(P)$ exactly once, therefore
$d_{\widetilde G}(x,y)={\rm length}(P)=||F(x)-F(y)||_1$ in this
case.\medskip

Case 2. $C>0$. Let $M_1$ be the number of cut edges in $I(P)$
which are used once in the walk $\pi(P)$. Let $M_2$ be the number
of edges of $I(P)$ which are in $2$-edge-connected components of
$I(P)$. Let $M_3$ be the number of cut edges of $I(P)$ which are
used twice in the walk $\pi(P)$. We have

\[||F(x)-F(y)||_1= M_1+M_2.\]

On the other hand,

\[d_{\widetilde G}(x,y)=M_1+M_2+2M_3.\]

In addition, by Observation \ref{O:card(i)}, and
\eqref{E:GirthMax} we have $M_2\ge C\cdot g(G)\ge C\cdot c\cdot
\diam(G)$. On the other hand, by Lemma \ref{L:length}, we have
$M_3\le N\cdot\diam(G)\le 3C\diam (G)$ (if $C\ge1$). Therefore
$M_3/M_2\le3/c$, and the quotient $d_{\widetilde
G}(x,y)/||F(x)-F(y)||$ is bounded above by a universal constant.
\end{proof}

\begin{remark}[Remark on applications to coarse embeddings] Since
$\ell_1$ admits a coarse embedding into a Hilbert space (see
\cite[Corollary 3.1]{Nao10}), Theorem \ref{T:main} implies that
the graphs $\widetilde G_n$ admit uniformly coarse embeddings into
a Hilbert space. Therefore, combining our Theorem \ref{T:main}
with a recent result of Willett \cite{Wil11}, we get more examples
of metric spaces with bounded geometry but without property A,
admitting coarse embeddings into a Hilbert space (first examples
of this type were found in \cite{AGS11+}). (It is worth mentioning
that without the bounded geometry condition such examples were
known earlier \cite{Now07}.)
\medskip

Also, it is worth mentioning that in \cite{Ost09} it was proved
that locally finite metric spaces which do not admit coarse
embeddings into a Hilbert space contain substructures which are
``locally expanding'' (see \cite{Ost09} for details). Our example,
as well as the example in \cite{AGS11+}, show that the converse it
false, since families of graphs with constant degree $\ge 3$ and
indefinitely growing girth are ``locally expanding'' in the sense
of \cite{Ost09}.
\end{remark}

The work on this paper started when the author was a participant
of the Workshop in Analysis and Probability at Texas A \&\ M
University.  The author would like to thank Florent Baudier, Ana
Khukhro, and Piotr Nowak for useful conversations related to the
subject of this paper during the workshop and to thank Nati Linial
and Doron Puder for helpful criticism of the first version of the
paper.

\end{large}

\begin{small}

\end{small}
\medskip

\noindent{\sc Department of Mathematics and Computer Science\\
St. John's University\\ 8000 Utopia Parkway, Queens, NY 11439,
USA}\\
e-mail: {\tt ostrovsm@stjohns.edu}


\begin{thebibliography}{AGS11+}

\bibitem[AL06]{AL06} A.~Amit, N.~Linial, Random lifts of graphs: edge
expansion, {\it Combin. Probab. Comput.}, {\bf 15} (2006),
317--332.

\bibitem[AGS11+]{AGS11+} G.~Arzhantseva, E.~Guentner, J.~\v Spakula,
Coarse non-amenability and coarse embeddings, {\tt
arXiv:1101.1993}

\bibitem[Bol78]{Bol78} B.~Bollob\'as, {\it Extremal graph theory},
Academic Press, London New York, 1978.

\bibitem[BM08]{BM08} J.\,A.~Bondy, U.\,S.\,R.~Murty, {\it Graph
theory}, Graduate Texts in Mathematics, {\bf 244}, Springer, New
York, 2008.

\bibitem[DL06]{DL06} Y.~Drier, N.~Linial, Minors in lifts of graphs, {\it Random
Structures Algorithms}, {\bf 29} (2006), 208--225

\bibitem[GT77]{GT77} J.\,L.~Gross, T.\,W.~Tucker, Generating all graph
coverings by permutation voltage assignments, {\it Discrete
Math.}, {\bf 18} (1977), 273--283.

\bibitem[GT87]{GT87} J.\,L.~Gross, T.\,W.~Tucker, {\it Topological
graph theory}, Wiley-Interscience Series in Discrete Mathematics
and Optimization, John Wiley \&\ Sons, Inc., New York, 1987.

\bibitem[Imr84]{Imr84} W.~Imrich,
Explicit construction of regular graphs without small cycles, {\it
Combinatorica}, {\bf 4} (1984), no. 1, 53--59.

\bibitem[Lin02]{Lin02} N.~Linial, Finite metric spaces--combinatorics,
geometry and algorithms, in: {\it Proceedings of the International
Congress of Mathematicians}, Vol. {\bf III} (Beijing, 2002),
573--586, Higher Ed. Press, Beijing, 2002; {\tt
arXiv:math.CO/0304466}.

\bibitem[LLR95]{LLR95} N.~Linial, E.~London, Y.~Rabinovich,
The geometry of graphs and some of its algorithmic applications,
{\it Combinatorica}, {\bf 15} (1995), no. 2, 215--245.

\bibitem[LMN02]{LMN02} N.~Linial, A.~Magen, A.~Naor, Girth and Euclidean distortion, {\it Geom. Funct. Anal.},
{\bf 12} (2002), 380--394.

\bibitem[LPS88]{LPS88} A.~Lubotzky, R.~Phillips, P.~Sarnak, Ramanujan
graphs, {\it Combinatorica}, {\bf 8} (1988), 261--277.

\bibitem[Mar82]{Mar82} G.\,A.~Margulis,
Explicit constructions of graphs without short cycles and low
density codes, {\it Combinatorica}, {\bf 2} (1982), no. 1, 71--78.

\bibitem[Mat02]{Mat02} J.~Matou\v sek, {\it Lectures on Discrete
Geometry}, Springer-Verlag, New York, 2002.

\bibitem[Mat05]{Mat05} J.~Matou\v sek, Embedding finite metric
spaces into normed spaces, an updated version of Chapter 15 of the
book \cite{Mat02}, prepared for a Japanese translation published
in 2005. English version is available at {\tt
http://kam.mff.cuni.cz/$\sim$matousek/}

\bibitem[Mat10]{Mat10} J.~Matou\v sek (Editor), Open problems
on embeddings of finite metric spaces, last revision: July 2010,
jointly with Assaf Naor; can be downloaded from {\tt
http://kam.mff.cuni.cz/$\sim$matousek/}

\bibitem[Nao10]{Nao10} A.~Naor, $L_1$ embeddings of the Heisenberg group and
fast estimation of graph isoperimetry, in: {\it Proceedings of the
International Congress of Mathematicians}, Hyderabad India, 2010;
{\tt arXiv:1003.4261}.

\bibitem[Now07]{Now07} P.\,W.~Nowak, Coarsely embeddable metric spaces
without property A, {\it J. Funct. Anal.}, {\bf 252} (2007),
126--136.


\bibitem[Ost09]{Ost09} M.\,I.~Ostrovskii, Expansion properties of metric spaces not
admitting a coarse embedding into a Hilbert space, {\it Comptes
rendus de l'Academie bulgare des Sciences}, {\bf 62} (2009),
415--420; Expanded version: {\tt arXiv:0903.0607}


\bibitem[Wil11]{Wil11} R.~Willett, Property A and graphs with large girth,
{\it preprint}, 2011.

\end{thebibliography}
\end{document}